\documentclass{article}
\usepackage{pinlabel}
\usepackage{graphicx}
\usepackage{amsmath}
\usepackage{amssymb}
\usepackage{amsthm}
\usepackage{multirow}
\usepackage{color}
\usepackage{soul}  

\input epsf.tex

\newcommand{\pic}[2]{\raisebox{-.5\height}
{\includegraphics[scale=#2]{#1}}}

\def\TalfaTalfaPrima{\pic{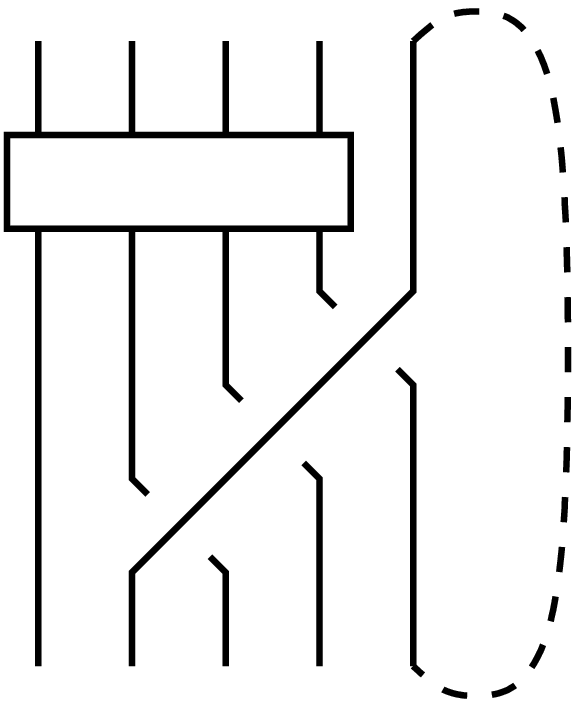}{.350}}
\def\TalfaTbetaEstrella{\pic{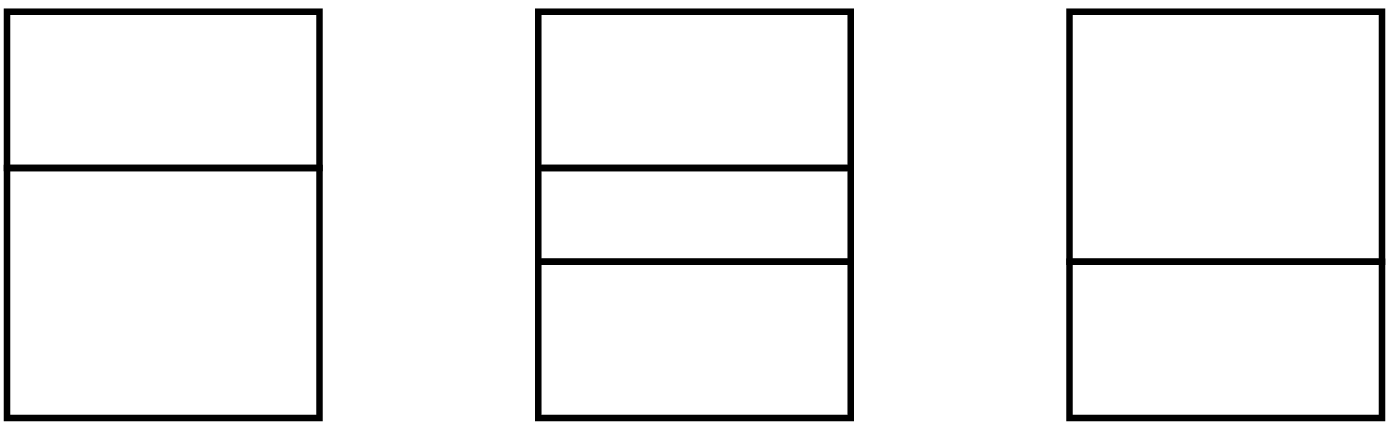}{.350}}
\def\NoSimpleTalfaTbetaEstrella{\pic{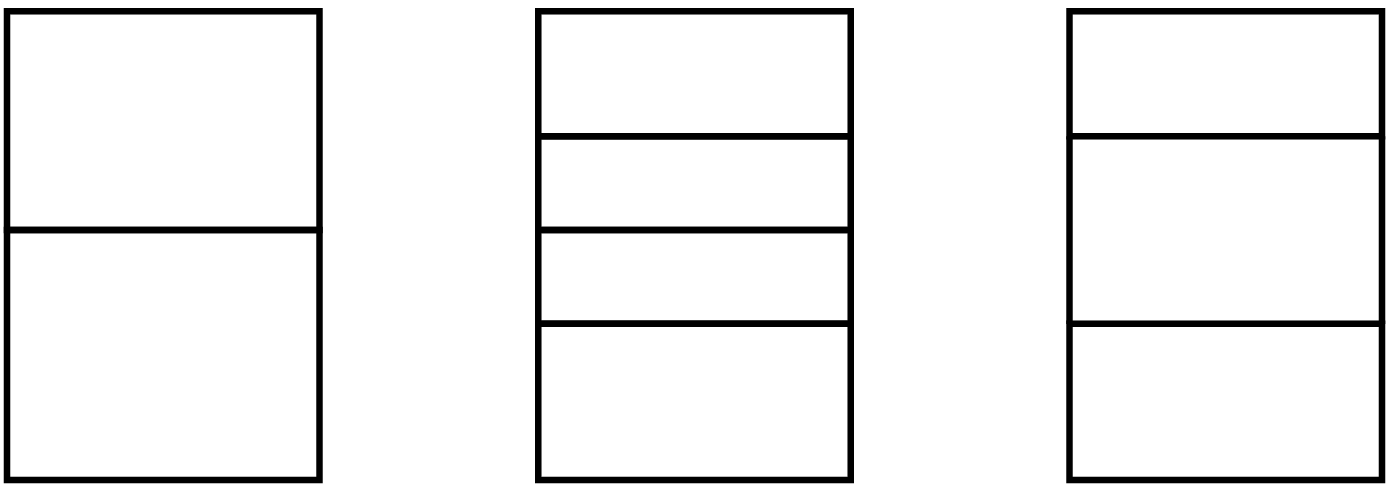}{.350}}
\def\arbol{\pic{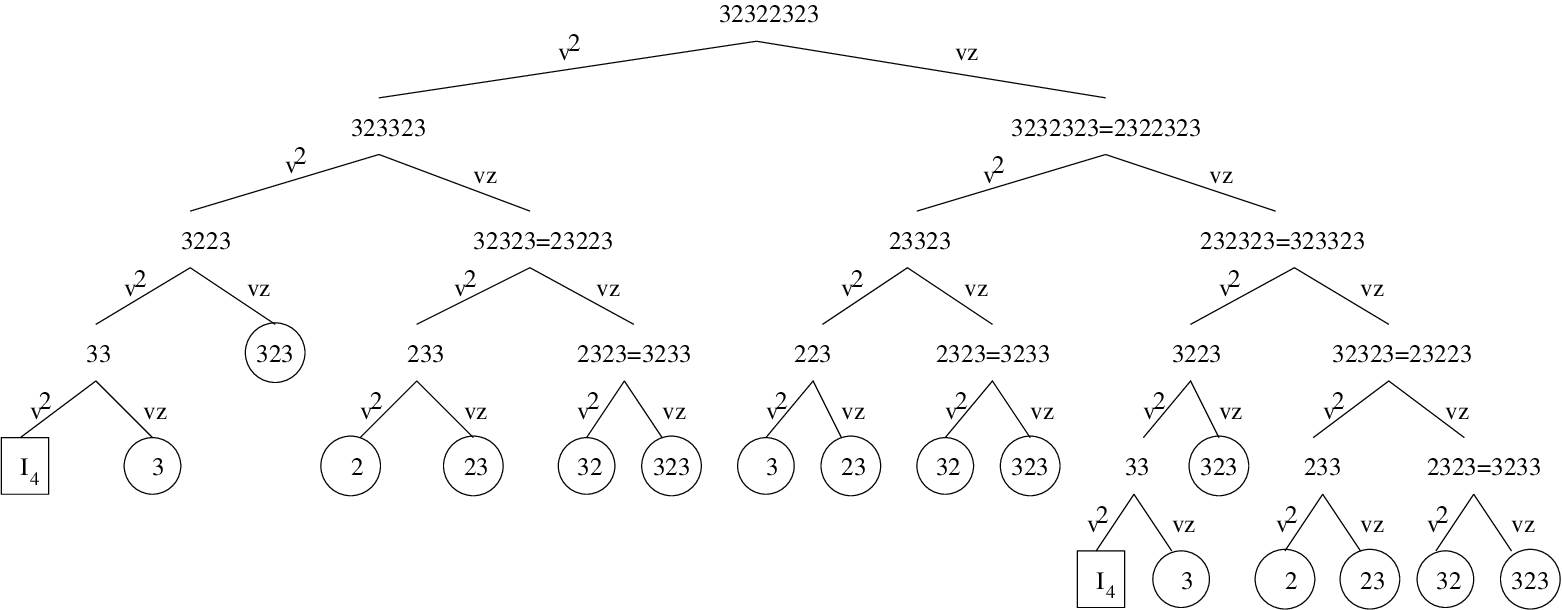}{.450}}

\newtheorem{theorem}{Theorem}
\newtheorem{proposition}[theorem]{Proposition}
\newtheorem{corollary}[theorem]{Corollary}

\renewenvironment{proof}[1][Proof]{\textit{#1.} }
{\hfill \rule{0.5em}{0.5em}}

\begin{document}

\title{Closures of positive braids and the Morton-Franks-Williams inequality}
\author{J. Gonz\'alez-Meneses\footnote{Both authors partially supported by MTM2010-19355 and FEDER. First author partially supported by FQM-P09-5112 and the Australian Research Council's Discovery Projects funding scheme (project number DP1094072).} \ and P. M. G. Manch\'on }
\date{6 August 2013}
\maketitle

\begin{abstract}
We study the Morton-Franks-Williams inequality for closures of simple braids (also known as positive permutation braids). This allows to prove, in a simple way, that the set of simple braids is a orthonormal basis for the inner product of the Hecke algebra of the braid group defined by K\'alm\'an, who first obtained this result by using an interesting connection with Contact Topology.

We also introduce a new technique to study the Homflypt polynomial for closures of positive braids, namely resolution trees whose leaves are simple braids. In terms of these simple resolution trees, we characterize closed positive braids for which the Morton-Franks-Williams inequality is strict. In particular, we determine explicitly the positive braid words on three strands whose closures have braid index three.
\end{abstract}

\section{Introduction} \label{introS}
Let $P_L(v,z)\in \Bbb{Z}[v^{\pm 1},z^{\pm 1}]$ be the two-variable Homflypt polynomial, isotopy invariant of oriented links with normalization $P_{\raisebox{-.8mm}{\epsfysize.10in \epsffile{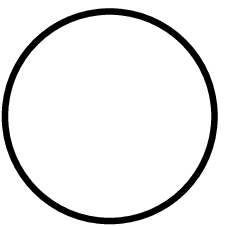}}}(v,z)=1$ and determined by the Homflypt skein relation
$$v^{-1}P_{\raisebox{-.8mm}{\epsfysize.15in \epsffile{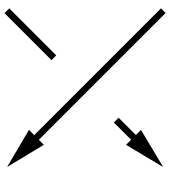}}}(v,z)
-vP_{\raisebox{-.8mm}{\epsfysize.15in \epsffile{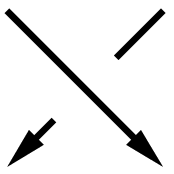}}}(v,z)
=zP_{\raisebox{-.8mm}{\epsfysize.15in \epsffile{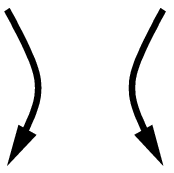}}}(v,z).$$
In the following we will use the notation $P(L)$ instead of $P_L(v,z)$. Note that, for braids, the Homflypt skein relation can be written as $v^{-1}\sigma_i-v\sigma_i^{-1}=z$, or equivalently, as the quadratic skein relation $\sigma_i^2=vz\sigma_i+v^2$.

We first recall and fix terminology about the Morton-Franks-Williams (MFW) bounds and inequalities. If $L=\widehat{b}$ is the closure of a braid $b \in B_n$ with $n$ strands and writhe $w=\textnormal{wr}(b)$, then $w-n+1\leq \partial_v^-(P(L))$ and $\partial_v^+(P(L)) \leq w+n-1$ are the known MFW lower and upper inequalities~\cite{Morton,FW}, where $\partial_v^-(P(L))$ (resp. $\partial_v^+(P(L))$) is the lowest (resp. highest) $v$-degree of $P(L)$. We refer to $w-n+1$ (resp. $w+n-1$) as the MFW lower (resp. upper) bound of $b$. It follows that, if we define
$$
  MFW(L)=\frac{\textnormal{span}_v(P(L))}{2}+1=\frac{\partial_v^+(P(L))-\partial_v^-(P(L))}{2}+1,
$$
we have the celebrated MFW inequality $MFW(L)\leq n$. In particular $MFW(L)\leq s(L)$ where $s(L)$ is the braid index (or Seifert circle index) of $L$.

Let $a,b \in B_n$ be two braids with $n$ strands. Then $\langle a,b \rangle _R$ is, by definition, the coefficient of $v^{w+n-1}$ in the two-variable polynomial $(-z)^{n-1}P(\widehat{ab^*})$, where $b^*$ is the reverse braid of $b$ and $w={\rm wr }(ab^*)$. This product can be extended to the whole of the Hecke algebra $H_n(z)$, obtaining a  symmetric bilinear form. This is the inner product introduced by K\'alm\'an in \cite{Kalman}. Recall that $H_n(z)$ can be seen as the linear combinations of braids in $B_n$ with coefficients in $\Bbb{Z}[z^{\pm 1}]$, quotiented by the Homflypt skein relation with $v=1$.

Given a permutation $\alpha \in S_n$ on $\{ 1, 2, \dots , n\}$, there is exactly one positive braid~$T_{\alpha}$ which determines the permutation $\alpha$ on its endpoints, and such that every two strands of it cross at most once. The braid $T_{\alpha}$ is said to be the simple braid associated to $\alpha$ (originally called positive permutation braid in \cite{MortonElrifai}). Note that ${\rm wr}(T_{\alpha})=l(\alpha )$, the length of the permutation $\alpha$. We will write $T_{\alpha}^*$ for $(T_{\alpha})^*$. It is a well-known result that the set of simple braids on $n$ strands is a basis of $H_n(z)$. Moreover, the main theorem in \cite{Kalman} states that it is an orthonormal basis for the above inner product. The original proof is based on Contact Topology: it constructs a Legendrian representative of the link $\widehat{T_{\alpha}T_{\beta}^*}$, and uses a result by Rutherford~\cite{Rutherford} that relates the ruling polynomial of a front projection of a Legendrian link with its Homflypt polynomial.

In this paper we relate all the above notions, namely we study how simple braids behave with respect to the MFW inequalities, and we apply the obtained results to K\'alm\'an's inner product, and to closures of positive braids on three strands.

More precisely, in Section~\ref{simpleS} we prove that, among all the closures of simple braids, the MFW upper bound is reached only for the closure of the identity braid. This is used in Section~\ref{KalmanS} to give a simple proof of K\'alm\'an's result: the set of simple braids is an orthonormal basis for K\'alm\'an's inner product. In particular, our proof contains implicitly an algorithm for calculating this inner product.

Further, in Section~\ref{resolutionS} we introduce the notion of simple resolution trees, as positive resolution trees whose leaves are simple braids. By using them, we will obtain a characterization of the closed positive braids for which the MFW inequality is sharp (Theorem~\ref{tree} and Corollary~\ref{transformations}). Note that, when working with closures of positive braids, the MFW lower bound is always reached (a fact that we easily reproved by using again simple resolution trees), hence the MFW inequality is sharp if and only if the MFW upper bound is reached. In particular, this technique allows us to determine explicitly in Section~\ref{3braidS} which positive braid words on three strands have closures of braid index three.

{\bf Acknowledgements:} Part of this work was done during a stay of the first author at the Centre de Recerca Matem\`atica (CRM) in Bellaterra (Barcelona, Spain) and at the Department of Applied Mathematics-EUITI, Universidad Polit\'ecnica de Madrid (Spain), and also during a stay of the second author at the Department of Algebra, Universidad de Sevilla (Spain). We thank these institutions for their hospitality.

\section{MFW inequality for simple braids}\label{simpleS}

In this section we show the key result in this paper: the MFW upper bound is reached, among closures of simple braids, only for the identity braid.

\begin{proposition}\label{MFWforSimples}
Let $\alpha \in S_n$ be a permutation with length $w=l(\alpha )$. Then
$\partial_v^+(P(\widehat{T_{\alpha}}))=w+n-1$ if and only if $\alpha= {\rm id}$, and the coefficient of $v^{w+n-1}$ in $P(\widehat{T_{\rm id}})$ is $(-z)^{1-n}$.
\end{proposition}

\begin{proof}
If $\alpha ={\rm id} \in S_n$ then $w=0$, $T_{\alpha}=1_n=||\stackrel{n}{\dots}|$ and $\widehat{T_{\alpha}}$ is a collection of $n$ unlinked trivial knots
$\raisebox{-.8mm}{\epsfysize.15in \epsffile{TrivialKnot.eps}}\raisebox{-.8mm}{\epsfysize.15in \epsffile{TrivialKnot.eps}}\stackrel{n}{\dots}\raisebox{-.8mm}{\epsfysize.15in \epsffile{TrivialKnot.eps}}$ with Homflypt polynomial
$$
P(\widehat{T_{\alpha}})=\delta ^{n-1}=\left( \frac{v^{-1}-v}{z}\right)^{n-1}= z^{1-n}v^{1-n} + \dots +(-z)^{1-n}v^{n-1}.
$$
So the result holds for $\alpha={\rm id}$, even in the extreme case when $n=1$.

We will prove the result by induction on $n$, the number of strands of the braid~$T_{\alpha}$.
As the trivial braid is the only braid on $1$ strand, we have already shown the case $n=1$.

Assume now the statement for $\alpha \in S_k$ with $k = 1, \dots , n-1$ and suppose $\alpha \in S_n$. Consider the inclusion $i:S_{n-1} \hookrightarrow S_n \ \ \omega \mapsto \omega \otimes 1$. We distinguish two cases:

    $\bullet $ If $\alpha \in S_n \setminus i(S_{n-1})$, there is a unique permutation $\alpha ' \in S_{n-1}$ and a unique natural number $k<n$ such that $\alpha = \alpha ' s_{n-1} s_{n-2} \dots s_{k}$. It turns out that $T_{\alpha}=T_{\alpha '}\sigma_{n-1} \sigma_{n-2} \dots \sigma_{k}$, as shown in Figure~\ref{TalfaTalfaPrimaF} (see for example \cite{BraidGroups}, page 167).
\begin{figure}[ht!]
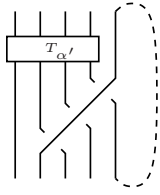

\labellist
\tiny
\pinlabel {$T_{\alpha '}$} at 60 148
\endlabellist
\begin{center}
  \TalfaTalfaPrima
  \end{center}
\caption{$T_{\alpha}=T_{\alpha '}\sigma_{n-1} \sigma_{n-2} \dots \sigma_{k}$, with $n=5$ and $k=2$}\label{TalfaTalfaPrimaF}
\end{figure}

Let $b=T_{\alpha '}\sigma_{n-2} \dots \sigma_{k} \in B_{n-1}$. Clearly $\widehat{T_{\alpha}}=\widehat{b}$, hence by the MFW upper inequality applied to the Homflypt polynomial of $\widehat{b}$, one has:
$$
\partial_v^+(P(\widehat{T_{\alpha}}))
=\partial_v^+(P(\widehat{b})) \leq (w-1)+(n-1)-1=w+n-3 < w+n-1.
$$

$\bullet $ If $\alpha =\alpha ' \otimes 1$ with $\alpha ' \in S_{n-1}$, then $T_{\alpha}=T_{\alpha '} \otimes 1$, $\widehat{T_{\alpha}}=\widehat{T_{\alpha '}} \sqcup \raisebox{-.8mm}{\epsfysize.15in \epsffile{TrivialKnot.eps}}$ and $P(\widehat{T_{\alpha}})=\delta P(\widehat{T_{\alpha '}})$. As we have already shown the result when $\alpha$ is trivial, we can assume that $\alpha\neq {\rm id}$ and then $\alpha'\neq {\rm id}$. Since $\delta=\frac{v^{-1}-v}{z}$, applying the induction hypothesis to $T_{\alpha'}$ which has $n-1$ strands and $w$ crossings, it follows that
$$
\partial_v^+(P(\widehat{T_{\alpha}}))=\partial_v^+(P(\widehat{T_{\alpha '}}))+1 < (w+(n-1)-1)+1 = w+n-1.
$$
\end{proof}

At this point, one could ask for an analogous result for the MFW lower inequality for closures of simple braids. However, it is known that, for closed positive
braids (and simple braids are positive braids) the MFW lower bound is always reached (see for example \cite{KalmanMeridian}, comment after Example 1$\cdot$8). In spite of this, we will give in Section~\ref{resolutionS} a direct proof of this fact, working with simple resolution trees.

\section{Inner products and the Homflypt skein relation} \label{KalmanS}

Recall, from the introduction, the inner product $\langle \cdot \:, \cdot \rangle _R$ defined by K\'alm\'an on the Hecke algebra $H_n(z)$. The following result was first obtained by K\'alm\'an~\cite{Kalman}, who proved it by using an interesting connection with Contact Topology (more details in the introduction). Here we give a simple proof of it, based on the Homflypt skein relation and on properties of the simple braids.

\begin{theorem}(K\'alm\'an) The set of simple braids $\{ T_{\alpha} \}_{\alpha \in S_n }$ is an orthonormal basis for $\langle \cdot \:, \cdot \rangle _R$.
\end{theorem}
\begin{proof}
We want to prove that, for any permutations $\alpha, \beta \in S_n$,
$$
\langle T_{\alpha},T_{\beta}\rangle _R=
\left\{
\begin{array}{cl}
1 & \text{ if } \beta =\alpha , \\
0 & \text{ otherwise. }
\end{array}
\right.
$$
This is equivalent to show that for all $\alpha, \beta \in S_n$, the coefficient of $v^{w+n-1}$ in $P(\widehat{T_\alpha T_\beta^*})$ is $(-z)^{1-n}$ if $\beta =\alpha$, and $0$ otherwise, where $w={\rm wr }(T_{\alpha}T_{\beta}^*)={\rm wr }(T_{\alpha}) + {\rm wr }(T_{\beta}) = l(\alpha )+l(\beta)$. Note that, by the MFW inequality, to say that the coefficient of $v^{w+n-1}$ in $P(\widehat{T_{\alpha}T_{\beta}^*})$ is zero is equivalent to say that $\partial_v^+(P(\widehat{T_{\alpha}T_{\beta}^*}))<w+n-1$.

The proof is by induction on the length $l(\beta)$ of the permutation $\beta$. If $l(\beta )=0$, then $\beta ={\rm id}$, $T_{\beta}=1_n \in B_n$, $\widehat{T_{\alpha}T_{\beta}^*}=\widehat{T_{\alpha}}$ and $w=l(\alpha)$. Then the result follows from Proposition~\ref{MFWforSimples}. Assume now that $l(\beta ) \geq 1$. Let $\beta =\kappa s_i$ with $l(\beta )=l(\kappa )+1$, hence $T_{\beta}=T_{\kappa }\sigma_i$ and $T_{\beta }^*=\sigma_i T_{\kappa }^*$.

\begin{enumerate}
  \item If $T_{\alpha}\sigma_i$ is a simple braid (equivalently $T_{\alpha}\sigma_i=T_{\alpha s_i}$), then we will see that $\alpha \not= \beta$ and $\langle T_\alpha, T_\beta \rangle_R=0$. Indeed, $\alpha =\beta$ would imply $T_\kappa \sigma_i\sigma_i=T_\beta \sigma_i =~T_\alpha \sigma_i$ to be a simple braid, a contradiction since in $T_\kappa \sigma_i\sigma_i$ the strands ending in positions $i$ and $i+1$ cross at least twice. In particular, $\alpha \not= \kappa s_i$ hence $\alpha s_i \not= \kappa$. Then (see Figure~\ref{FigureCaseOne}) $T_{\alpha}T_{\beta}^*=T_{\alpha } \sigma_i T_{\kappa}^*=T_{\alpha s_i} T_{\kappa}^*$ hence $\langle T_{\alpha}, T_{\beta}\rangle _R = \langle T_{\alpha s_i}, T_{\kappa}\rangle _R=0$ where we have applied induction in the last equality since $l(\kappa )<l(\beta )$.
\begin{figure}[ht!]
\labellist
\pinlabel {$T_{\alpha}$} at 50 95
\pinlabel {$T_{\alpha}$} at 200 95
\pinlabel {$T_{\alpha s_i}$} at 352 85
\pinlabel {$=$} at 123 60
\pinlabel {$\sigma_i$} at 200 60
\pinlabel {$=$} at 273 60
\pinlabel {$T_{\beta}^*$} at 50 40
\pinlabel {$T_{\kappa}^*$} at 200 25
\pinlabel {$T_{\kappa}^*$} at 350 25
\endlabellist
\begin{center}
  \TalfaTbetaEstrella
  \end{center}
\caption{Case $T_{\alpha}\sigma_i$ simple: $T_{\alpha}T_{\beta}^*=T_{\alpha}\sigma_iT_{\kappa}^*=T_{\alpha s_i}T_{\kappa}^*$}\label{FigureCaseOne}
\end{figure}

\item If $T_{\alpha}\sigma_i$ is a non-simple braid, then $l(\alpha s_i)=l(\alpha)-1$ and there exists a reduced expression $\alpha =\alpha_1 s_i$ of $\alpha$ ending with $s_i$ and $T_{\alpha}=T_{\alpha_1}\sigma_i$.

\begin{figure}[ht!]
\labellist
\pinlabel {$T_{\alpha}$} at 50 102
\pinlabel {$T_{\beta}^*$} at 50 37
\pinlabel {$T_{\alpha_1}$} at 200 118
\pinlabel {$T_{\alpha_1}$} at 350 118
\pinlabel {$=$} at 123 70
\pinlabel {$=$} at 273 70
\pinlabel {$\sigma_i$} at 200 85
\pinlabel {$\sigma_i$} at 200 60
\pinlabel {$\sigma_i^2$} at 350 73
\pinlabel {$T_{\kappa}^*$} at 200 25
\pinlabel {$T_{\kappa}^*$} at 350 25
\endlabellist
\begin{center}
  \NoSimpleTalfaTbetaEstrella
  \end{center}
\caption{Case $T_{\alpha}\sigma_i$ non-simple: $T_{\alpha}T_{\beta}^*=T_{\alpha_1}\sigma_i^2T_{\kappa}^*$} \label{FigureCaseTwo}
\end{figure}

Then $T_{\alpha}T_{\beta}^*=T_{\alpha_1}\sigma_i^2T_{\kappa}^*$ (see Figure~\ref{FigureCaseTwo}) and, by the quadratic relation $\sigma_i^2=vz\sigma_i+v^2$,
$$
P(\widehat{T_{\alpha}T_{\beta}^*})
= P(\widehat{T_{\alpha_1}\sigma_i ^2T_{\kappa}^*})
= vzP(\widehat{T_{\alpha}T_{\kappa}^*})
+v^2P(\widehat{T_{\alpha_1}T_{\kappa}^*}).
$$

Multiplying the above equality by $(-z)^{n-1}$ and considering the coefficients of $v^{w-n+1}$, it follows that:
$$
\langle T_{\alpha},T_{\beta}\rangle_R
=z\langle T_{\alpha},T_{\kappa}\rangle_R + \langle T_{\alpha_1},T_{\kappa}\rangle_R.
$$

Note that $\alpha_1=\kappa \Leftrightarrow \alpha_1s_i=\kappa s_i \Leftrightarrow \alpha=\beta$. We finally distinguish two cases; induction will be applicable since $l(\kappa)<l(\beta)$:

\vspace{0.1cm}

\noindent $\bullet$ Assume $\alpha = \beta$. Then $\alpha \not= \kappa$ and $\alpha_1 =\kappa$, hence
$$
\langle T_{\alpha},T_{\beta}\rangle_R
=z\langle T_{\alpha},T_{\kappa}\rangle_R + \langle T_{\alpha_1},T_{\kappa}\rangle_R
=z\!\cdot \!0+1=1.
$$

\noindent $\bullet$ Assume $\alpha \not= \beta$. In particular $\alpha_1 \not= \kappa$. Moreover, $\alpha \not= \kappa$ since, otherwise, $T_{\alpha}\sigma_i=T_{\kappa}\sigma_i=T_{\beta}$ would be a simple braid. Hence
$$
\langle T_{\alpha},T_{\beta}\rangle_R
=z\langle T_{\alpha},T_{\kappa}\rangle_R + \langle T_{\alpha_1},T_{\kappa}\rangle_R
=z\!\cdot \!0+0=0.
$$
\end{enumerate}\vspace{-0.4cm}\end{proof}

\section{Positive braids and the Morton-Frank-Williams inequality}\label{resolutionS}

\noindent Suppose that $b$ is a positive braid whose closure is the oriented link $L$. Based on the quadratic relation $\sigma_i^2=vz\sigma_i+v^2$, in order to calculate the Homflypt polynomial of $L$ we can construct a resolution tree of $b$. That is, a binary tree with root $b$ and where each ramification has the following form, with $P$ and $Q$ positive braids:

\begin{figure}[ht!]
\begin{center}
\setlength{\unitlength}{1mm}
\begin{picture}(30, 20)
  \put(10,16){$P\sigma_i^2 Q$}
  \put(5,9){$v^2$}
  \put(20,9){$vz$}
  \put(0,0){$PQ$}
  \put(21,0){$P\sigma_i Q$}
  \put(14, 14){\line(1, -1){10}}
  \put(14, 14){\line(-1, -1){10}}
\end{picture}
\end{center}
\vspace{-0.3cm}\caption{Parent $P\sigma_i^2Q$, left child $PQ$ and right child $P\sigma_i Q$}
\end{figure}
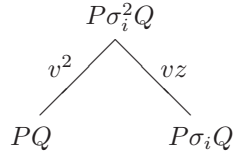

It is known (see the proof of Theorem~\ref{tree}) that simple braids are precisely those positive braids which cannot be written as $P\sigma_i^2Q$, with $P$ and $Q$ positive braids. This means that if a simple braid appears in a resolution tree, it must necessarily be a leaf.  A resolution tree is called {\it simple} if all the leaves are simple braids. As far as we know, simple resolution trees have not been considered yet; positive resolution trees have been used for example in \cite{Nakamura}.

As an example, we show in Figure~\ref{EPP} a simple resolution tree for the braid $b=32322323$ (meaning $\sigma_3\sigma_2\sigma_3\sigma_2\sigma_2\sigma_3\sigma_2\sigma_3$),  with writhe $w=8$ and $n=4$ strands. This tree shows in an explicit way that the Homflypt polynomial of~$\widehat{b}$ is a combination of Homflypt polynomials of closures of simple braids, with coefficients in $\mathbb N[z,v]$ given by the product of the edge labels in the path going from each leaf to the root $b$. Collecting the leaves which correspond to the same simple braid, in this example we obtain
$$
 \begin{array}{rcl}
  P(\widehat{b}) & = & (1+z^2)\cdot  v^8 \cdot P(\widehat{I_4}) \\
  && \\
  & + &(z+z^3)\cdot v^7 \cdot P(\widehat{\sigma_2}) + (2z+z^3) \cdot v^7 \cdot  P(\widehat{\sigma_3}) \\
&& \\
   &   + & (2z^2+z^4) \cdot v^{6} \cdot P(\widehat{\sigma_2\sigma_3})+
(2z^2+z^4) \cdot v^{6} \cdot P(\widehat{\sigma_3\sigma_2}) \\
&& \\
& + &(z+3z^3+z^5) \cdot v^{5} \cdot P(\widehat{\sigma_3\sigma_2\sigma_3}).
 \end{array}
$$

\begin{figure}[ht!]
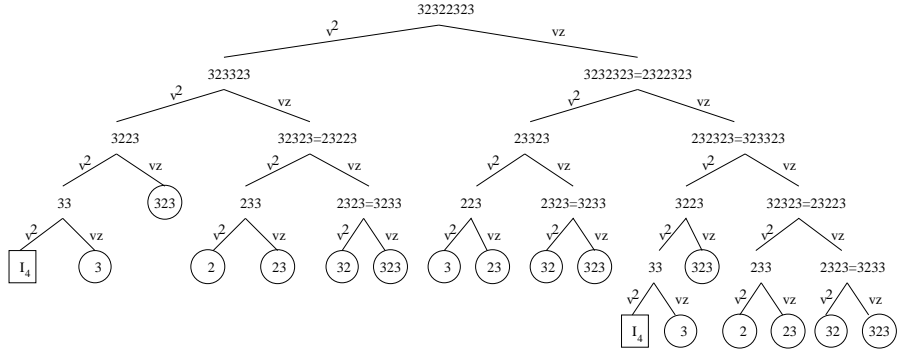

\begin{center}
\arbol
\end{center}
\caption{A simple resolution tree for the braid $b=32322323$}\label{EPP}
\end{figure}

We now prove that simple resolution trees always exist and that we can directly see, from a simple resolution tree, whether a closed positive braid reaches the MFW upper bound.

\begin{theorem}\label{tree}
Let $L=\widehat{b}$ be a link obtained as the closure of a positive braid~$b$ of $n$ strands and writhe $w$. Then $b$ admits a simple resolution tree and, moreover, the MFW upper bound is sharp for $L$, that is, $\partial_v^+(P(L))=w+n-1$, if and only if at least one leaf in this simple resolution tree is the identity braid.
\end{theorem}
\begin{proof}
That any positive braid has a simple resolution tree follows from the following well known fact: a positive braid $\beta$ is not simple if and only if we can decompose it as $\beta=P\sigma_i^2 Q$, with $P$ and $Q$ positive braids (see, for example, Lemma 2.5 and following remark in \cite{MortonElrifai}). As relations in the braid group are homogeneous, the lengths of the braids $PQ$ and $P\sigma_i Q$ are strictly smaller than the length of~$\beta$. Therefore, starting with the root $b$, we can iteratively decompose every node which is not simple into two smaller nodes. Clearly this process terminates, yielding a simple resolution tree for $b$.

Let $T_{\alpha_1},\ldots,T_{\alpha_k}$ be the (not necessarily distinct) simple braids corresponding to the $k$ leaves of a simple resolution tree of $b$. For $i=1,\ldots,k$, let $z^{a_i}v^{w-l(\alpha_i)}$ be the monomial obtained by multiplying the edge labels of the path that goes from the leaf $T_{\alpha_i}$ to the root $b$. Note that $a_i$ is the number of right children in this path. Then
\begin{equation}\label{Eq}
P(L) = \sum_{i=1}^{k}{z^{a_i}\cdot v^{w-l(\alpha_i)} \cdot P(\widehat{T_{\alpha_i}})}.
\end{equation}

Since ${\rm wr}(T_{\alpha_i})=l(\alpha_i)$, by Proposition~\ref{MFWforSimples} the highest $v$-degree of each summand is $(w-l(\alpha_i))+({\rm wr}(T_{\alpha_i})+n-1)=w+n-1$ if and only if $\alpha_i={\rm id}$. Then, if no leaf is the identity braid, $\partial_v^+(P(L))<w+n-1$. Reciprocally, if at least one leaf of the simple resolution tree is trivial, then the coefficient of $v^{w-n+1}$ in $P(L)$ is, again by Proposition~\ref{MFWforSimples}, the sum of the monomials $z^{a_j}(-z)^{1-n}$ that correspond to $\alpha_j ={\rm id}$, a sum which is obviously nonzero.
\end{proof}

As stated in the introduction, it is well known that the MFW lower inequality is actually an equality for the closure of any positive braid (see for example~\cite{KalmanMeridian}, comment after Example 1$\cdot$8). Here we reprove this result by making use of Equation~(\ref{Eq}) derived from a simple resolution tree, and following the steps in the proof of Proposition~\ref{MFWforSimples}. Rather than give complete details, we prefer to explain some historical remarks about the positiveness of the Homflypt polynomial. Recall that a nonzero (Laurent) polynomial in $z$ is said to be {\it positive} if all its coefficients are nonnegative. Answering positively a question by V. F. R. Jones, Cromwell and Morton \cite{MortonCromwell} proved that, for positive links, the evaluation of the Homflypt polynomial $P(L)(v,z)$ in any $v \in (0,1)$ provides a positive Laurent polynomial in $z$. If $v=1$ we obtain the Conway polynomial, also positive except that it can be zero if the original link is split.

\begin{proposition}\label{PositiveBraidsMFWlowerbound}
Let $L=\widehat{b}$ be a link obtained as the closure of a positive braid~$b$ of $n$ strands and writhe $w$. Then the MFW lower inequality is sharp for~$L$, that is, $\partial_v^-(P(L))=w-n+1$. Moreover, the coefficient of $v^{w-n+1}$ in $P(L)$ is a positive Laurent polynomial in $z$.
\end{proposition}

\begin{proof}
Following a double induction, first on the number of strands and then on the writhe, we will see that the coefficient $q_L(z)$ of $v^{w-n+1}$ in $P(L)(v,z)$ is a positive Laurent polynomial in the variable $z$. If $n=1$, then $w=0$ and the closure of the braid is the trivial knot with polynomial $1$, so the result holds.

For $n>1$ we follow the steps in the proof of Proposition~\ref{MFWforSimples} to see first that the result is true for any simple braid $T_\alpha$ with $n$ strands. If $\alpha \in S_n \setminus i(S_{n-1})$ then $\widehat{T_\alpha}=\widehat{d}$ for a positive braid $d$ with $n-1$ strands and writhe $w-1$, as given in the proof of Proposition~\ref{MFWforSimples}. Thus $P(\widehat{T_\alpha})=P(\widehat{d})$ and $(w-1)-(n-1)+1=w-n+1$. Since $d$ is positive (although non-simple) and has less than $n$ strands, induction can be applied. If $\alpha \in i(S_{n-1})$ then $\alpha =\alpha' \otimes 1$ and $P(\widehat{T_\alpha})=\delta P(\widehat{T_{\alpha '}})$ where $T_{\alpha '}$ is simple, with the same writhe as $T_\alpha$ and one less strand (again, see the proof of Proposition~\ref{MFWforSimples}). Clearly, $q_L(z)=\frac{1}{z}q_{\widehat{T_{\alpha '}}}(z)$, so the result holds for every simple braid with $n$ strands.

Finally, once we have proved the result for the closure of simple braids with $n$ strands, the result for a positive braid with $n$ strands follows from considering Equation~(\ref{Eq}), derived from a simple resolution tree.
\end{proof}

According to Proposition~\ref{PositiveBraidsMFWlowerbound}, for closures of positive braids the MFW inequality is sharp if and only if the MFW upper bound is reached. Then the following result is a nice consequence of Theorem~\ref{tree}:

\begin{corollary}\label{transformations}
The MFW inequality is sharp for a closed positive braid if and only if one (hence all) of its braid word representatives can be obtained from the empty word by a finite sequence of transformations of the following types:
\begin{enumerate}
  \item Inserting $\sigma_i^2$ for some $i=1, \dots , n-1$,
  \item doubling a letter $\sigma_i$ for some $i=1, \dots , n-1$, and
  \item applying positive braid relations.
\end{enumerate}
\end{corollary}
\begin{proof}
Starting with the empty word, a sequence of the above transformations builds a branch of a simple resolution tree for the corresponding positive braid. Since the leaf of this branch is the identity, the MFW upper bound is sharp according to Theorem~\ref{tree}.

Reciprocally, suppose that $L=\widehat{b}$ reaches the MFW upper bound, and construct a simple resolution tree for $b$. By Theorem~\ref{tree} at least one of its leaves is the identity. Ascending in the tree from such a leaf provides the sequence of transformations of the above types which define a braid word for $b$.
\end{proof}

We now enumerate some examples which can be deduced from Corollary~\ref{transformations}:

\begin{corollary}\label{listado}
Let $w$ be a positive word representing a braid $b$. Then the MFW inequality is sharp for the oriented link $L=\widehat{b}$ if the word $w$ is in the following list:
\begin{enumerate}
  \item Words which are product of positive powers of the generators, where all the exponents are greater than or equal to~two, that is, $w=\prod_k \sigma_{i_k}^{e_k}$ with $e_k \geq 2$ for all $k$. For example, $\sigma_3^2\sigma_2^5\sigma_1^2\sigma_2^3\sigma_3^3$.
\item Even positive palindromic braid words, that is, positive words with an even number of letters that reads the same backwards as forwards. For example, $\sigma_3\sigma_2\sigma_1^2\sigma_2\sigma_3$.
\item Any word of the form $uw_0v$ where $u, v$ are positive words and $w_0$ is any positive word representing the square of the half twist $\Delta \in B_n$.
\end{enumerate}
\end{corollary}
\begin{proof}
Words in the first item can be obtained by a finite number of transformations of type 1 and 2 in Corollary~\ref{transformations}. Words in the second item can be obtained by a finite number of transformations of type 1 in Corollary~\ref{transformations}.

To prove the statement for words in the third item, we first recall that the half twist or Garside element $\Delta \in B_n$ can be represented by two words which are the reverse of each other:
\begin{eqnarray*}
   \Delta & = & \sigma_1(\sigma_2\sigma_1)\cdots (\sigma_{n-2}\cdots \sigma_1)(\sigma_{n-1}\cdots \sigma_1) \\
          & = &  (\sigma_1\cdots \sigma_{n-1})(\sigma_1\cdots \sigma_{n-2}) \cdots (\sigma_1\sigma_2)\sigma_1
\end{eqnarray*}
This means that $\Delta^2$ can be represented by an even positive palindromic braid word $w_0$, known already to be in the list. Note that any other positive word representing $\Delta^2$ is also in the list, since it can be obtained from $w_0$ by positive braid relations (transformation of type 3 in Corollary~\ref{transformations}).

The half twist $\Delta$ can be represented by a positive word ending (or starting) with any generator $\sigma_i$~\cite{MortonElrifai}. It follows easily the same for $w_0$; by positive braid relations $w_0$ can be transformed into a positive word $w_i'$ (resp. $w_i$) that starts (resp. ends) with $\sigma_i$. Then, if $v=\sigma_{i_1}\cdots \sigma_{i_k}$, by positive braid relations we transform $w_0$ into $w_{i_k}$, and then double the last letter $\sigma_{i_k}$ by a transformation of type 2. Next we apply positive braid relations to transform $w_{i_k}$ into $w_{i_{k-1}}$, and double its last letter to obtain $w_{i_{k-1}}\sigma_{i_{k-1}}\sigma_{i_k}$. Iterating this process, we finally obtain $w_{i_1}\sigma_{i_1}\sigma_{i_2}\cdots \sigma_{i_k}$, which by positive braid relations can be transformed into $w_0v$. Finally we repeat the whole process on the left, using the equivalent words $w_i'$, to obtain $uw_0v$.
\end{proof}

It is probably worth to rewrite the last item in Corollary~\ref{listado} (which was already shown in~\cite[Corollary 2.4]{FW}) with other words:
\begin{corollary}\label{Delta2C}
Let $a, b$ be two positive braids. Then the MFW inequality is sharp for the closure of the braid $a\Delta^2 b$.
\end{corollary}

Using the terminology from Garside theory~\cite{MortonElrifai}, the above result means that the MFW inequality is sharp for positive braids of {\it infimum} at least two. Therefore, the MFW inequality can be strict only for positive braids whose infimum is zero or one.

Recall the celebrated lower bound for the braid index $s(L)$ of an oriented link~$L$, defined as
$$
  MFW(L)=\frac{\textnormal{span}_v(P(L))}{2}+1=\frac{(\partial_v^+(P(L))-\partial_v^-(P(L)))}{2}+1
$$
(see \cite{Morton}, \cite{FW}). In \cite{FW} Franks and Williams conjectured that, for a link which is closure of a positive braid, $MFW(L)=s(L)$. In \cite{MortonShort} Morton and Short showed a counterexample: for $L=\widehat{b}$ with $b=\sigma_{3}\sigma_{2}\sigma_{1}\sigma_{3}\sigma_{2}^2\sigma_{1}\sigma_{3}\sigma_{2}^2\sigma_{1}\sigma_{3}\sigma_{2} \in B_4$ we have $MFW(L)=3$ and $s(L)=4$. However, it is known  that $MFW(L)=2$ if and only if $s(L)=2$, if and only if $L$ is a torus link $T(2,n)$ for certain $n\geq 2$ (see \cite[Theorem 1.2]{Nakamura}). We prove the following result:

\begin{proposition}
Let $L$ be an oriented link. Then $MFW(L)=s(L)$ if there exists a positive braid $b$ with $L=\widehat{b}$ and $b$ admits a simple resolution tree where at least one leaf is the identity braid. In this case, if $b\in B_n$, then $s(L)=n$.
\end{proposition}
\begin{proof}
Assume that $L=\widehat{b}$ where $b$ is a positive braid with $n$ strands and writhe $w$, and the identity braid with $n$ strands is one of the leaves of a simple resolution tree of $b$. By Proposition~\ref{PositiveBraidsMFWlowerbound} we have $\partial_v^-(P(L))=w-n+1$ and by Theorem~\ref{tree} $\partial_v^+(P(L))=w+n-1$. In particular $MFW(L)=n$. Since $MFW(L) \leq s(L) \leq n$, the results follows.
\end{proof}

Even if we restrict our attention to the oriented links which are closed positive braids, the converse result is not clear to us, since there are oriented links which are closures of positive braids, but with no positive braid representations of minimal number of strands \cite[Theorem 1]{Stoimenow}. The example exhibited by Stoimenow has braid index $s(L)=4$. We do not know if there are examples with $s(L)=3$.

\section{Positive braids on three strands}\label{3braidS}

We end this paper with a study of positive braid words on three strands. More precisely, we will study the braid index of their closures.

Clearly, the links of braid index one and two are precisely the trivial knot and the torus links $T(2,k)$ for $k\in \mathbb Z\setminus\{-1,1\}$. It is well known (see~\cite{BM93} and also~\cite[Theorem 1.1]{BM08}) that a  braid with three strands closes to a link whose braid index is smaller than three (one of the above) if and only if it is conjugate to $\sigma_1^k\sigma_2^{\pm 1}$ for some $k\in \mathbb Z$. Hence, knowing how to solve the conjugacy problem in $B_3$ one can determine the braid index of a closed braid with three strands.

The next result, which uses the techniques introduced in this paper, avoids the need to use the conjugacy problem in the case of {\it positive} braids on three strands, as we give a complete list of positive words whose closures have braid index smaller than three.

\begin{theorem}~\label{3braidsT}
Let $w$ be a positive word in $\sigma_1, \sigma_2$, and let $b$ be the braid on three strands represented by $w$. Then the braid index of $\widehat{b}$ is smaller than three if and only if $w$ is, up to cyclic permutation, one of the following words:
\begin{enumerate}

 \item $\sigma_1\sigma_2^p$\quad or \quad  $\sigma_2 \sigma_1^{p}$, \quad  for $p\geq 0$.

 \item $\sigma_1\sigma_2\sigma_1^p\sigma_2^q$  \quad or \quad $\sigma_2\sigma_1 \sigma_2^p \sigma_1^q$, \quad for $p,q>0$.
\end{enumerate}
\end{theorem}

\begin{proof}
It is known~\cite[Proposition 3.1]{Nakamura} that if $L$ is the closure of a positive braid~$b$ on $n=3$ strands, then $MFW(L)=s(L)$, the braid index of $L$. And clearly, for braids on three strands, that the MFW inequality is sharp means exactly that $MFW(L)=3$. We now examine the different possibilities.

If $w$ is the trivial word the result holds trivially, as the trivial link with three components has braid index three.

Suppose that $w=\sigma_i^k$ only involves one of the generators $\sigma_1$ or $\sigma_2$. If $k=1$, $w$ is in the list above and clearly the braid index of $\widehat{b}$ is two. If $k>1$ then $w$ is not in the list (even considering cyclic permutation) and $w$ can be obtained by inserting $\sigma_i^2$ and then doubling $\sigma_i$ as many times as needed. By Corollary~\ref{transformations}, the result follows.

We can then assume that $w$ involves $\sigma_1$ and $\sigma_2$ and, after a cyclic permutation of its letters, that there are exponents $e_i>0$ for $i=1,\ldots,2k$ with
$$
   w= \sigma_1^{e_1}\sigma_2^{e_2}\sigma_1^{e_3}\sigma_2^{e_4}\cdots \sigma_1^{e_{2k-1}}\sigma_2^{e_{2k}}.
$$

Suppose $k\geq 3$. We will produce $w$ from the trivial word going up in a simple resolution tree (that is, applying the transformations from Corollary~\ref{transformations}). By Corollary~\ref{Delta2C}, we can produce $\Delta^2\sigma_1^{e_7}\sigma_2^{e_8}\cdots \sigma_1^{e_{2k-1}}\sigma_2^{e_{2k}}$. That is, we can produce $\sigma_1\sigma_2\sigma_1\sigma_2\sigma_1\sigma_2\sigma_1^{e_7}\sigma_2^{e_8}\cdots \sigma_1^{e_{2k-1}}\sigma_2^{e_{2k}}$. Now doubling the letters in $\sigma_1\sigma_2\sigma_1\sigma_2\sigma_1\sigma_2$ as many times as needed, one obtains $w$. This implies, from Corollary~\ref{transformations}, that if $k\geq 3$ the braid index of $b$ is three.

It remains to study the cases $k=1$ and $k=2$.

Suppose $k=1$, so $w=\sigma_1^{e_1}\sigma_2^{e_2}$. If both $e_i>1$ the word is not in the list and the braid index is three by Corollary~\ref{listado}, so we can assume that either $e_1=1$ or $e_2=1$. If $e_1=1$ then $w=\sigma_1\sigma_2^{e_2}$, which clearly has braid index smaller than three, as it corresponds to a stabilization of a braid on two strands. The same happens if $e_2=1$, in which case $w=\sigma_1^{e_1}\sigma_2$ is equivalent to $\sigma_2\sigma_1^{e_1}$ up to cyclic permutation of its letters.

Suppose finally that $k=2$, so $w=\sigma_1^{e_1}\sigma_2^{e_2}\sigma_1^{e_3}\sigma_2^{e_4}$. Let us suppose that $w=\sigma_1\sigma_2\sigma_1^{p}\sigma_2^{q}$ with $p,q>0$. Then $b=\sigma_1\sigma_2\sigma_1^p\sigma_2^q = \sigma_2\sigma_1\sigma_2\sigma_1^{p-1}\sigma_2^q$ which is conjugate to $\sigma_1\sigma_2\sigma_1^{p-1}\sigma_2^{q+1}$. Repeating this process, we see that $b$ is conjugate to $\sigma_1\sigma_2 \sigma_1^{0}\sigma_2^{p+q} = \sigma_1\sigma_2^{p+q+1}$, so $\widehat{b}$ has braid index smaller than three, by the previous case. Similarly, if $w=\sigma_2\sigma_1\sigma_2^p\sigma_1^q$ with $p,q>0$, the braid index of $\widehat{b}$ is smaller than three.

We know from Corollary~\ref{listado} that if $e_i>1$ for $i=1, 2, 3, 4$ the braid index of~$\widehat{b}$ is three. Hence, up to cyclic permutation of the letters, and exchange of letters $\sigma_1$ and $\sigma_2$ (which preserves the braid index), the only remaining case is $w=\sigma_1\sigma_2^{e_2}\sigma_1^{e_3}\sigma_2^{e_4}$, with $e_2,e_4>1$. But in this case $w$ can be obtained from the trivial word going up in a simple resolution tree as follows: first we produce~$\sigma_1^2$, then we insert $\sigma_2^2$ twice to produce $\sigma_1\sigma_2^2\sigma_1\sigma_2^2$, and finally we double $\sigma_2$ and the second $\sigma_1$ as many times as needed (recall that $e_2$ and $e_4$ are greater than one). This implies that, in this case, the braid index of~$\widehat{b}$ is three.

Therefore, the only words which represent a braid whose closure has braid index smaller than three are, up to cyclic permutation of their letters, the ones in the statement.
\end{proof}

A straightforward consequence of the above result is the following, which could also be derived from~\cite[Theorem 1.1]{BM08}.

\begin{corollary}
Given a positive braid $b$ on three strands, the braid index of $\widehat b$ is smaller than three if and only if $b$ is conjugate to $\sigma_1^p \sigma_2$ for some $p\geq 0$.
\end{corollary}

\begin{proof}
This result follows immediately from Theorem~\ref{3braidsT}, as all braids appearing in its statement are conjugate to $\sigma_1^p \sigma_2$ for some $p\geq 0$. More precisely, one has:
$$
 \sigma_2^{-1}\Delta^{-1}(\sigma_1\sigma_2^p)\Delta\sigma_2  = \sigma_1^p \sigma_2,
$$
$$
 \sigma_2^{-1}(\sigma_2\sigma_1^p)\sigma_2  = \sigma_1^p \sigma_2,
$$
$$
 \sigma_2^{-1}\Delta^{-1}\sigma_2^{-p}(\sigma_1\sigma_2\sigma_1^p\sigma_2^q)\sigma_2^p\Delta\sigma_2 = \sigma_1^{p+q+1}\sigma_2,
$$
$$
 \sigma_2^{-1}\sigma_1^{-p}(\sigma_2\sigma_1\sigma_2^p\sigma_1^q)\sigma_1^p\sigma_2 = \sigma_1^{p+q+1}\sigma_2.
$$
Conversely, every braid conjugated to $\sigma_1^p\sigma_2$ for some $p\geq 0$ has the same closure as $\sigma_1^p\sigma_2$, which has braid index smaller than three as it is the stabilization of the 2-strands braid $\sigma_1^p$.
\end{proof}

\begin{tabular}{ll}
Juan Gonz\'alez-Meneses & Pedro M. Gonz\'alez Manch\'on \\ & \\
Departamento de \'Algebra & Departamento de Matem\'atica Aplicada \\
Facultad de Matem\'aticas & EUITI-Universidad Polit\'ecnica de Madrid \\
Instituto de Matem\'aticas (IMUS) & Ronda de Valencia 3 \\
Universidad de Sevilla  & 28012 Madrid (Spain) \\
Apdo. 1160 &  {\it pedro.gmanchon@upm.es} \\
41080 Sevilla (Spain) & \\
{\it meneses@us.es} &
\end{tabular}

\end{document}